\documentclass[10pt]{article}
\usepackage{makeidx}
\makeindex
\usepackage{amsmath}
\usepackage{amssymb}
\usepackage{amsthm}
\usepackage{fullpage}

\begin{document}
\newtheorem{theorem}{Theorem}
\newtheorem{lemma}{Lemma}
\newtheorem{proposition}{Proposition}
\newtheorem{corollary}{Corollary}
\newtheorem{example}{Example}

\title{Invasion and non-invasion on a time-periodic domain}
\date{}
\author{Jane Allwright}
\maketitle

\section*{Abstract}
For a two-species reaction-diffusion-competition system on a domain that translates at constant speed and/or whose boundary varies periodically with time, we prove sufficient conditions such that one species can, and cannot, invade an established population of the other. These results extend those of Potapov and Lewis (2004) to the periodic case, and to more general reaction terms.

\section{Introduction}
In the paper \cite{PotLew} Potapov and Lewis study a two-species competition system of the form
\begin{equation}\label{eq_PotLew_1}
\frac{\partial u_1}{\partial t} = D_1\frac{\partial^2 u_1}{\partial x^2}+ f_1\left( u_1, u_2\right)
\end{equation}
\begin{equation}
\frac{\partial u_2}{\partial t} = D_2\frac{\partial^2 u_2}{\partial x^2}+ f_2\left( u_1, u_2\right)
\end{equation}
on an interval which is either stationary or translating with a constant speed $c$. Their reaction-competition terms have the Lotka-Volterra form
\begin{equation}\label{eq_f1f2_LotkaVolterra}
f_1(u_1,u_2)=u_1(r_1 - \alpha_{11}u_1 - \alpha_{12}u_2), \qquad 
f_2(u_1,u_2)=u_2(r_2 - \alpha_{21}u_1 - \alpha_{22}u_2).
\end{equation}
Among other results, they prove some sufficient conditions such that the second species $u_2$ can, or cannot, invade a stationary solution $U_1(x)$ of the first equation (which exists provided that $r_1$ is large enough). This amounts to finding conditions under which the stationary state $(U_1,0)$ is either unstable or stable with respect to introducing a small amount of $u_2$. However they give such conditions only for the case of a stationary domain ($c=0$). 

Here we shall extend these invasibility and non-invasibility results of \cite{PotLew} in several ways. Namely, we extend the work to domains in higher dimensions, to non-zero speeds $c$, to domains whose boundaries vary periodically with time, and to more general forms of the reaction and competition terms $f_1$, $f_2$. We consider
\begin{align}\label{eq_psi_1}
\frac{\partial \psi_1}{\partial t} &= D_1 \nabla^2 \psi_1+ f_1\left( \psi_1, \psi_2\right) \\
\frac{\partial \psi_2}{\partial t} &= D_2 \nabla^2 \psi_2+ f_2\left( \psi_1, \psi_2\right)
\end{align}
on domains of the form
\begin{equation}
\Omega(t)=\{ x\in\mathbb{R}^N: x-ct \in\Omega_*(t)\}, \qquad \Omega_*(t)\equiv \Omega_*(t+T)
\end{equation}
where $c \in \mathbb{R}^N$ is a constant vector and $\Omega_*(t) \subset \mathbb{R}^N$ is smooth, bounded, connected and $T$-periodic domain.
We note that this includes the case where $\Omega_*(t) \equiv \Omega_0$ is a constant domain and $\Omega(t)=\Omega_0+ct$ translates at a constant velocity $c$.
We assume both $\psi_1$ and $\psi_2$ are non-negative and satisfy zero Dirichlet boundary conditions:
\begin{equation}
\psi_1(x,t)=\psi_2(x,t)=0 \qquad\textrm{for } x\in\partial\Omega(t).
\end{equation}
Under certain assumptions on $\Omega(t)$ and $f_1$ which will be made precise, it is known \cite{JA-PhD, JA3} that there exists a unique positive $T$-periodic solution $\psi_1^{*}(x,t)$ to the first equation, \eqref{eq_psi_1}. In this paper, inspired by \cite{PotLew}, we derive conditions under which $(\psi_1^{*}(x,t),0)$ is stable or unstable with respect to adding a small amount of $\psi_2$. Our approach is based upon the principal eigenvalues of suitably constructed periodic-parabolic eigenvalue problems.

In the context of a habitat region whose boundary varies periodically with time or moves with a constant drift, this provides sufficient conditions such that a species can, or cannot, invade an established $T$-periodic population of another species. By a `successful invasion' here we mean that the introduced (invading) species persists rather than its population decaying to zero. Further work is required to deduce the long-time outcome, namely whether the invader replaces the established species or whether there is a state of co-existence.

\section{Assumptions}
Let us state our assumptions on the domain and the nonlinear terms.

Regarding $\Omega_*(t)$, we shall assume there is a one-to-one mapping $h(\cdot ,t): \overline{\Omega_*(t)}\rightarrow \overline{\Omega_0}$ which transforms $\Omega_*(t)$ into a bounded, connected reference domain $\Omega_0$ with sufficiently smooth boundary (at least $C^{2+\varepsilon}$ for some $\varepsilon>0$), and such that the change of variables $\xi=h(x-ct,t)$
transforms the operator $\frac{\partial}{\partial t} - D_n\nabla^2$ on $\Omega(t)$ into $\frac{\partial}{\partial t} -\mathcal{L}_n(\xi,t)$ on $\Omega_0$. Here (for $n=1, 2$)
\begin{equation}
\mathcal{L}_n(\xi,t) u  = D_n \sum_{i,j}a_{ij}(\xi,t)\frac{\partial^2 u}{\partial \xi_i\partial \xi_j} +\sum_j \left(b_j(\xi,t)+D_n d_j(\xi,t)\right)\frac{\partial u}{\partial \xi_j} \qquad\textrm{for } \xi\in \Omega_0,
\end{equation}
where
\begin{equation}
a_{ij}(\xi,t)= \sum_{k}\left(\frac{\partial h_i}{\partial x_k}\frac{\partial h_j}{\partial x_k}\right), \qquad
b_j(\xi,t)= \sum_k c_k\frac{\partial h_j}{\partial x_k} -\frac{\partial h_j}{\partial t}, \qquad
d_j(\xi,t)=\nabla^2 h_j.
\end{equation}
We assume that the map $h$ is such that the coefficients $a_{ij}$, $b_j$, $d_j$ belong to $C^{\alpha,\alpha/2}(\overline{\Omega_0}\times[0,T])$ for some $\alpha>0$, and that $a_{ij}$ is uniformly elliptic. Thus, letting $u_1(\xi,t)=\psi_1(x,t)$ and $u_2(\xi,t)=\psi_2(x,t)$ we obtain a problem on $\Omega_0$ of the form
\begin{equation}
\frac{\partial u_1}{\partial t} = \mathcal{L}_1(\xi,t) u_1 + f_1(u_1, u_2)
\end{equation}
\begin{equation}\label{eq_u2}
\frac{\partial u_2}{\partial t} = \mathcal{L}_2(\xi,t) u_2 + f_2(u_1, u_2)
\end{equation}
with $u_1(\xi,t)=u_2(\xi,t)=0$ on $\partial\Omega_0$.
\begin{example}
Consider an interval $A(t)+ct<x<A(t)+ct+L(t)$ where $L(t)>0$, $c$ is a constant, and where $L(t)$ and $A(t)$ are smooth and $T$-periodic. Letting $\xi=\left(\frac{x-A(t)-ct}{L(t)}\right)L_0$, the problem becomes
\begin{equation}
\frac{\partial u_1}{\partial t} = D_1 \frac{L_0^2}{L(t)^2} \frac{\partial^2 u_1}{\partial \xi^2}+\left(\frac{(c+\dot{A}(t))L_0+\xi\dot{L}(t)}{L(t)}\right)\frac{\partial u_1}{\partial \xi} +f_1( u_1,u_2)
\end{equation}
\begin{equation}
\frac{\partial u_2}{\partial t} = D_2 \frac{L_0^2}{L(t)^2} \frac{\partial^2 u_2}{\partial \xi^2}+\left(\frac{(c+\dot{A}(t))L_0+\xi\dot{L}(t)}{L(t)}\right)\frac{\partial u_2}{\partial \xi} +f_2( u_1,u_2)
\end{equation}
for $0<\xi<L_0$, and with $u_1=u_2=0$ at $\xi=0$ and $\xi=L_0$.
\end{example}

Regarding the nonlinear reaction and competition terms, we shall assume that the function $f_1$ is continuous, and let $F_1(u):= f_1(u,0)$. In particular
\begin{equation}\label{assumptions_on_f1}
f_1(u_1,u_2)= F_1(u_1)+o(1) \qquad \textrm{as } u_2 \rightarrow0.
\end{equation}
We assume that $F_1$ is Lipschitz continuous, differentiable at $0$, and satisfies the following conditions for some $K_1>0$:
\begin{align}\label{assumptions_on_F_1_1}
& F_1(0)=F_1(K_1)=0, \quad F_1'(0)=r_1>0, \quad \frac{F_1(u)}{u}\textrm{ is non-increasing on }u>0.
\end{align}
Under these assumptions, we can write
\begin{align}\label{assumptions_on_F_1_2}
F_1(u)&=u\left(r_1-h_1(u)\right)
\end{align}
where $h_1(0)=0$ and $h_1$ is continuous and non-decreasing for $u\geq0$.

Regarding the function $f_2$, we shall assume that we can write
\begin{align}\label{assumptions_on_f2}
f_2(u_1,u_2)&=u_2\left(r_2-g_2(u_1)\right)+o(u_2) \qquad \textrm{as } u_2 \rightarrow0
\end{align}
where $r_2>0$ and $g_2$ is a continuous function with $g_2(0)=0$ and $g_2(u_1)\geq 0$ for $u_1\geq 0$.

For certain results we will also write $F_2(u):= f_2(0,u)$ and assume that
$F_2$ is Lipschitz continuous, differentiable at $0$, and satisfies the following conditions for some $K_2>0$:
\begin{align}\label{assumptions_on_F_2_1}
& F_2(0)=F_2(K_2)=0, \quad F_2'(0)=r_2>0, \quad \frac{F_2(u)}{u}\textrm{ is non-increasing on }u>0.
\end{align}
Under these assumptions, we can write
\begin{align}\label{assumptions_on_F_2_2}
F_2(u)&=u\left(r_2-h_2(u)\right)
\end{align}
where $h_2(0)=0$ and $h_2$ is continuous and non-decreasing for $u\geq0$.

\section{Approach}
Since $\Omega_*(t)$ is periodic with period $T$, the map $h$ and the coefficients of $\mathcal{L}_1$, $\mathcal{L}_2$ are also $T$-periodic in $t$.
By Theorem 1 of Castro and Lazer \cite{CasLaz} there exist a unique $\mu_1$ and a function $\phi_1(\xi,t)$ such that
\begin{equation}
\frac{\partial \phi_1}{\partial t} - \mathcal{L}_1 \phi_1 = \mu_1\phi_1  \quad \textrm{for } \xi\in\Omega_0,\ t\in\mathbb{R}
\end{equation}
\begin{equation}
\phi_1=0 \quad \textrm{on } \partial\Omega_0, \qquad
\phi_1>0 \quad \textrm{in } \Omega_0, \qquad
\phi_1(\xi,t)\equiv\phi_1(\xi,t+T) .
\end{equation}
This function $\phi_1$ is unique up to scaling \cite[Theorem 1]{CasLaz}, and is called the principal periodic eigenfunction, while $\mu_1$ is called the principal periodic eigenvalue.
Likewise, there exist a unique $\mu_2$ and $\phi_2(\xi,t)$ such that
\begin{equation}
\frac{\partial \phi_2}{\partial t} - \mathcal{L}_2 \phi_2 = \mu_2\phi_2  \quad \textrm{for } \xi\in\Omega_0,\ t\in\mathbb{R}
\end{equation}
\begin{equation}
\phi_2=0 \quad \textrm{on } \partial\Omega_0, \qquad
\phi_2>0 \quad \textrm{in } \Omega_0, \qquad
\phi_2(\xi,t)\equiv\phi_2(\xi,t+T) .
\end{equation}

From the work on a single equation \cite{JA-PhD, JA3}, it is known that if $r_1>\mu_1$ then there exists a unique positive $T$-periodic solution $0 \leq u_1^{*}(\xi,t) \leq K_1$ to the nonlinear problem
\begin{equation}
\frac{\partial u_1^{*}}{\partial t} = \mathcal{L}_1u_1^{*} +F_1(u_1^{*})  \quad \textrm{for } \xi\in\Omega_0, \ t\in\mathbb{R}
\end{equation}
\begin{equation}
u_1^{*}=0 \quad \textrm{on } \partial\Omega_0, \qquad
u_1^{*}>0 \quad \textrm{in } \Omega_0, \qquad
u_1^{*}(\xi,t)\equiv u_1^{*}(\xi,t+T).
\end{equation}
Moreover, if $r_1>\mu_1$ then in a single species model (i.e. $u_2 \equiv 0$), $u_1$ converges uniformly to $u_1^{*}$ in the sense that as $n\rightarrow\infty$, $u_1(\xi,nT+t)$ converges in $C^{2,1}(\overline{\Omega_0}\times[0,T])$ to $u_1^{*}(\xi,t)$ \cite{JA-PhD, JA3}.

Similarly, if \eqref{assumptions_on_F_2_1}, \eqref{assumptions_on_F_2_2} are satisfied, and if $r_2>\mu_2$, then there exists a unique positive $T$-periodic solution $0 \leq u_2^{*}(\xi,t) \leq K_2$ to the nonlinear problem
\begin{equation}
\frac{\partial u_2^{*}}{\partial t} = \mathcal{L}_2u_2^{*} +F_2(u_2^{*})  \quad \textrm{for } \xi\in\Omega_0, \ t\in\mathbb{R}
\end{equation}
\begin{equation}
u_2^{*}=0 \quad \textrm{on } \partial\Omega_0, \qquad
u_2^{*}>0 \quad \textrm{in } \Omega_0, \qquad
u_2^{*}(\xi,t)\equiv u_2^{*}(\xi,t+T).
\end{equation}

Henceforth we shall assume that $r_1>\mu_1$ and consider the established positive $T$-periodic solution of $u_1$:
\begin{equation}
(u_1,u_2)=(u_1^{*}(\xi,t), 0).
\end{equation}
We are interested in conditions under which this is either stable (the species $u_2$ cannot invade) or unstable ($u_2$ can invade successfully). We shall follow the ideas of Potapov and Lewis \cite{PotLew}, and extend their invasibility and non-invasibility theorems to periodic domains (as described above) as well as those moving at constant speed.

As in \cite{PotLew} the approach is to introduce a small amount of $u_2$ and consider the linearised problem. The linearisation of \eqref{eq_u2} about the state $(u_1^{*}(\xi,t), 0)$ leads to the equation
\begin{equation}\label{eq_linearisation}
\frac{\partial u_2}{\partial t} = \mathcal{L}_2(\xi,t) u_2 + u_2\left(r_2-g_2 (u_1^{*}(\xi,t))\right)
\end{equation}
\begin{equation}
u_2(\xi,t)=0 \qquad\textrm{for } \xi\in\partial\Omega_0.
\end{equation}
Now, again using \cite{CasLaz}, we know that there is a principal periodic eigenvalue $\hat{\mu}$ and positive periodic eigenfunction $\phi(\xi,t)$ to the linear periodic-parabolic problem
\begin{equation}\label{eq_muhat}
\frac{\partial \phi}{\partial t} - \mathcal{L}_2(\xi,t) \phi +g_2(u_1^{*}(\xi,t))\phi = \hat{\mu}\phi  \quad \textrm{for } \xi\in\Omega_0,\ t\in\mathbb{R}
\end{equation}
\begin{equation}
\phi(\xi,t)=0 \quad \textrm{on } \partial\Omega_0, \qquad
\phi(\xi,t)>0 \quad \textrm{in } \Omega_0, \qquad
\phi(\xi,t)\equiv\phi(\xi,t+T) .
\end{equation}
Therefore $\phi(\xi,t)e^{(r_2-\hat{\mu})t}$ is a solution to the linearised problem
\eqref{eq_linearisation}.
Hence, by the comparison principle, we see that if $r_2> \hat{\mu}$ then the linearised $u_2$ grows in the presence of $u_1^{*}(\xi,t)$, whereas if $r_2< \hat{\mu}$ then $u_2$ will decay to zero.

\section{Invasibility and non-invasibility results}
Here we derive sufficient conditions guaranteeing either invasibility ($r_2> \hat{\mu}$) or non-invasibility ($r_2< \hat{\mu}$). We also give some examples. The first proposition is based on  \cite[Lemma 6.1]{PotLew}.
\begin{proposition}\label{proposition_r2_mu2_gu1}
Let $f_1$, $f_2$ satisfy assumptions \eqref{assumptions_on_f1}, \eqref{assumptions_on_F_1_1} and \eqref{assumptions_on_f2}. Let $\beta=\sup_{\Omega_0 \times[0, T]} g_2(u_1^{*})$ and suppose that $\beta>0$.
\begin{enumerate}
\item
If $r_2\geq \mu_2 +\beta$
then $u_2$ is able to invade $(u_1^{*}(\xi,t),0)$ successfully.
\item
If $r_2 \leq \mu_2$ then $u_2$ is not able to invade $(u_1^{*}(\xi,t),0)$.
\end{enumerate}
\end{proposition}
\begin{proof}
The assumptions imply that
$0\lneq  g_2(u_1^{*}(\xi,t)) \lneq  \beta$ and thus that
$\mu_2<\hat{\mu}<\mu_{\beta}$
where $\mu_{\beta}$ is the periodic principal eigenvalue of the equation
\begin{equation}
\frac{\partial \phi}{\partial t} - \mathcal{L}_2(\xi,t) \phi +\beta\phi = \mu_{\beta} \phi  \qquad \textrm{for } \xi\in\Omega_0,\ t\in\mathbb{R}.
\end{equation}
But this is precisely $\mu_{\beta}= \beta+ \mu_2$. So, $\mu_2<\hat{\mu}<\mu_{\beta}=\beta+ \mu_2$.
Consequently, if $r_2 \geq \mu_2 + \beta$ then $r_2>\hat{\mu}$, whereas if $r_2 \leq \mu_2$ then $r_2<\hat{\mu}$.
\end{proof}
We give two typical applications of Proposition \ref{proposition_r2_mu2_gu1}.
\begin{example}
Consider an interval moving at constant speed: $ct<x<ct+L_0$. The problem becomes
\begin{equation}\label{eq_c_u1}
\frac{\partial u_1}{\partial t} = D_1\frac{\partial^2 u_1}{\partial \xi^2} +c \frac{\partial u_1}{\partial \xi} + f_1(u_1,u_2)
\end{equation}
\begin{equation}\label{eq_c_u2}
\frac{\partial u_2}{\partial t} = D_2 \frac{\partial^2 u_2}{\partial \xi^2} +c \frac{\partial u_2}{\partial \xi} + f_2(u_1,u_2).
\end{equation}
The values of $\mu_1$, $\mu_2$ are known exactly: $\mu_1= \frac{D_1 \pi^2}{L_0^2}+\frac{c^2}{4D_1}$ and $\mu_2= \frac{D_2 \pi^2}{L_0^2}+\frac{c^2}{4D_2}$.
We assume that $f_1$, $f_2$ satisfy assumptions \eqref{assumptions_on_f1}, \eqref{assumptions_on_F_1_1} and \eqref{assumptions_on_f2}, and that
\begin{equation}\label{eq_cond_r1}
r_1> \frac{D_1 \pi^2}{L_0^2}+\frac{c^2}{4D_1}.
\end{equation}
Then, in the absence of $u_2$, the first species $u_1(\xi,t)$ converges to the unique positive stationary state $0 \lneq U_1(\xi) \leq K_1$ satisfying
\begin{equation}\label{eq_U1_ode}
D_1 U_1'' +cU_1' +F_1(U_1)=0, \qquad U_1(0)=U_1(L_0)=0 .
\end{equation}
Certainly $\sup_{[0,L_0]} g_2(U_1)\leq \hat{\beta}:=\sup_{[0, K_1]} g_2$.
Proposition \ref{proposition_r2_mu2_gu1} implies that if
$r_2 \leq \frac{D_2 \pi^2}{L_0^2}+\frac{c^2}{4D_2}$
then $u_2$ cannot invade the positive stationary state $U_1(\xi)$. However, if
\begin{equation}\label{eq_cond_r2_beta}
r_2 \geq \frac{D_2 \pi^2}{L_0^2}+\frac{c^2}{4D_2} +\hat{\beta}
\end{equation}
then $u_2$ can invade $U_1(\xi)$ successfully. Note in particular that if
\begin{equation}\label{eq_cond_possible}
r_1 > \frac{c^2}{4D_1}\qquad \textrm{and} \qquad
r_2 - \frac{c^2}{4D_2} > \hat{\beta}>0
\end{equation}
then conditions \eqref{eq_cond_r1} and \eqref{eq_cond_r2_beta} become
\begin{equation}
L_0>\pi\sqrt{\frac{D_1}{r_1-\frac{c^2}{4D_1}}}
\qquad\textrm{and} \qquad
L_0>\pi\sqrt{\frac{D_2}{r_2-\frac{c^2}{4D_2}-\hat{\beta}}} \ .
\end{equation}
Therefore if \eqref{eq_cond_possible} holds then $u_2$ can invade $U_1(\xi)$ for $L_0$ large enough.

In \cite{PotLew} Potapov and Lewis proved the above results for the case $c=0$ and for $f_1$, $f_2$ of the form \eqref{eq_f1f2_LotkaVolterra}.
\end{example}

\begin{example}
Consider an interval $A(t)<x<A(t)+L(t)$ where $A(t)$ and $L(t)$ are both $T$-periodic. Again let $\hat{\beta}=\sup_{[0, K_1]} g_2$. For a general $T$-periodic interval, the values of $\mu_1$ and $\mu_2$ are not known exactly but \cite{JA-PhD, JA3} provides upper and lower bounds on them. To use these bounds we must define
\begin{align}
\overline{Q}(t)=\max_{0\leq\eta\leq 1} \left(\frac{\eta^2\ddot{L}(t)L(t)}{2} + \eta \ddot{A}(t)L(t) \right),\nonumber\\
\underline{Q}(t)=-\min_{0\leq\eta\leq 1} \left( \frac{\eta^2\ddot{L}(t)L(t)}{2} + \eta \ddot{A}(t)L(t) \right).
\end{align}

We can conclude from Proposition \ref{proposition_r2_mu2_gu1}, together with the lower bounds for $\mu_2$ (see \cite{JA-PhD, JA3}), that if
\begin{equation}
r_2 \leq \max\left\{\frac{1}{T}\int_0^T  \frac{D\pi^2}{L(t)^2} dt, \  \frac{1}{T}\int_0^T \left(\frac{D\pi^2}{L(t)^2} +\frac{\dot{A}(t)^2}{4D} -\frac{\overline{Q}(t)}{2D}\right) dt \right\}
\end{equation}
then $r_2\leq \mu_2$ and so $u_2$ cannot invade $(u_1^{*}(\xi,t),0)$.

Using Proposition \ref{proposition_r2_mu2_gu1}, together with the upper bounds for $\mu_2$ (see \cite{JA-PhD, JA3}), we conclude that if
\begin{equation}\label{eq_r2_cond_periodic}
r_2 \geq \hat{\beta} + \frac{1}{T}\int_0^T \left(\frac{D\pi^2}{L(t)^2}+\frac{\dot{A}(t)^2}{4D}+\frac{\underline{Q}(t)}{2D}\right) dt,
\end{equation}
or, if
\begin{equation}
\min_{[0,T]}(A+L) - \max_{[0,T]} A >0  \qquad \textrm{and} \qquad
r_2 \geq \hat{\beta} +  \frac{D\pi^2}{(\min(A+L) - \max A)^2},
\end{equation}
then $r_2\geq \hat{\beta}+\mu_2$ and so $u_2$ is able to invade $(u_1^{*}(\xi,t),0)$ successfully.

Suppose that $L(t)=l\left(\frac{\omega t}{2\pi}\right)$ and $A(t)= a\left( \frac{\omega t}{2\pi}\right)$ for some $1$-periodic functions $l$, $a$. Note that if
\begin{equation}\label{eq_r2_cond_periodic_omega}
r_2 > \hat{\beta} + \int_0^1 \frac{D\pi^2}{l(s)^2} ds
\end{equation}
then \eqref{eq_r2_cond_periodic} will be satisfied for $\omega$ small enough.
\end{example}

Our second proposition is based on ideas from the proof of \cite[Theorem 6.1]{PotLew}.
\begin{proposition} \label{proposition_hu2_gu1}
Let $f_1$, $f_2$ satisfy assumptions \eqref{assumptions_on_f1}, \eqref{assumptions_on_F_1_1}, \eqref{assumptions_on_f2}, \eqref{assumptions_on_F_2_1} and \eqref{assumptions_on_F_2_2}.
Assume $r_2>\mu_2$ and let $u_2^{*}(\xi,t)$ be as above.
If $g_2(u_1^{*}) \lneq  h_2(u_2^{*})$ then $u_2$ can invade $u_1^{*}(\xi,t)$. 
If the opposite inequality holds then $u_2$ cannot invade $u_1^{*}(\xi,t)$.
\end{proposition}
\begin{proof}
If $g_2(u_1^{*}) \lneq  h_2(u_2^{*})$, then $\hat{\mu}< \mu_{*}$ where $\mu_{*}$ is the principal periodic eigenvalue of the equation
\begin{equation}\label{eq_h2_u2}
\frac{\partial \phi}{\partial t} - \mathcal{L}_2(\xi,t) \phi +h_2(u_2^{*}(\xi,t))\phi = \mu_{*} \phi  \qquad \textrm{for } \xi\in\Omega_0,\ t\in\mathbb{R}.
\end{equation}
(This follows by the same proof as \cite[Lemma 15.5]{Hess}; see also Section 2.5.2 of \cite{CantrellCosner}.)
But equation \eqref{eq_h2_u2} is satisfied by $u_2^{*}(\xi,t)$, with $\mu_{*}=r_2$, and so by uniqueness these must be the principal eigenfunction and eigenvalue. Therefore, $\hat{\mu}<\mu_{*} =r_2$. If the inequality is reversed then instead we get $\hat{\mu}> \mu_{*} =r_2$.
\end{proof}

Next we apply Proposition \ref{proposition_hu2_gu1}
for the case of an interval moving at constant speed $c$ (Corollary \ref{corollary_1}) and the case of a periodic interval (Corollary \ref{corollary_2}).
\begin{corollary}\label{corollary_1} (See \cite[Theorem 6.1]{PotLew}, where they use this method to prove the case $c=0$, $n=1$.)\\
Consider the problem \eqref{eq_c_u1}, \eqref{eq_c_u2}.
Let $f_1$, $f_2$ satisfy assumptions \eqref{assumptions_on_f1}, \eqref{assumptions_on_F_1_1}, \eqref{assumptions_on_F_1_2} \eqref{assumptions_on_f2}, \eqref{assumptions_on_F_2_1} and \eqref{assumptions_on_F_2_2}, and suppose that for some $n>0$,
\begin{equation}\label{eq_gh}
h_1(u)=(\hat{h}_1u)^n, \qquad g_2(u)=(\hat{g}_2u)^n, \qquad h_2(u)=(\hat{h}_2u)^n
\end{equation}
for positive constants $\hat{h}_1$, $\hat{g}_2$ and $\hat{h}_2$.
Assume that $r_1 >\frac{D_1 \pi^2}{L_0^2}+\frac{c^2}{4D_1}$, $r_2 > \frac{D_2 \pi^2}{L_0^2}+\frac{c^2}{4D_2}$, and let $U_1(\xi)$ be the positive stationary state satisfying \eqref{eq_U1_ode}, and $U_2(\xi)$ the positive stationary state satisfying
\begin{equation}
D_2 U_2'' +cU_2' +F_2(U_2)=0, \qquad U_2(0)=U_2(L_0)=0 .
\end{equation}
\begin{enumerate}
\item Suppose $c=0$.\\
If $\frac{r_2}{r_1} \geq \frac{D_2}{D_1}$ and $\left(\frac{r_2}{r_1}\right)^{\frac{1}{n}}\hat{h}_1 \geq \hat{g}_2$, and at least one of these inequalities is strict, then $u_2$ can invade $(U_1(\xi),0)$.\\
If $\frac{r_2}{r_1} \leq \frac{D_2}{D_1}$ and $\left(\frac{r_2}{r_1}\right)^{\frac{1}{n}}\hat{h}_1 \leq \hat{g}_2$, and at least one of these inequalities is strict, then $u_2$ cannot invade $(U_1(\xi),0)$.
\item Suppose $D_1=D_2$ (and $c$ may be zero or non-zero). \\
If $r_2 \geq r_1$ and $\left(\frac{r_2}{r_1}\right)^{\frac{1}{n}}\hat{h}_1 \geq \hat{g}_2$, and at least one of these inequalities is strict, then $u_2$ can invade $(U_1(\xi),0)$.\\
If $r_2 \leq r_1$ and $\left(\frac{r_2}{r_1}\right)^{\frac{1}{n}}\hat{h}_1 \leq \hat{g}_2$, and at least one of these inequalities is strict, then $u_2$ cannot invade $(U_1(\xi),0)$.
\item Suppose $c\neq 0$ and $D_1\neq D_2$.\\
If $r_2 -\frac{c^2}{4D_2} \geq \frac{D_2}{D_1} \left(r_1 -\frac{c^2}{4D_1}\right)$
and $\frac{\hat{g}_2}{\hat{h}_1} \leq \left(\frac{D_2}{D_1}\right)^{\frac{1}{n}} e^{-\frac{L_0}{2}\left\vert c (\frac{1}{D_2}-\frac{1}{D_1})\right\vert}$ then $u_2$ can invade $(U_1(\xi),0)$.\\
If $r_2 -\frac{c^2}{4D_2} \leq \frac{D_2}{D_1} \left(r_1 -\frac{c^2}{4D_1}\right)$ and $\frac{\hat{g}_2}{\hat{h}_1} \geq \left(\frac{D_2}{D_1}\right)^{\frac{1}{n}} e^{\frac{L_0}{2}\left\vert c (\frac{1}{D_2}-\frac{1}{D_1})\right\vert}$ then $u_2$ cannot invade $(U_1(\xi),0)$.
\end{enumerate}
\end{corollary}
\begin{proof}
We shall show that the conditions from Proposition \ref{proposition_hu2_gu1} hold.
Since $g_2$ and $h_2$ are given by \eqref{eq_gh}, what we need to show for the invasibility is that
\begin{equation}\label{eq_aim1_gU2_hU1}
\hat{g}_2 U_1 \leq \hat{h}_2 U_2
\end{equation}
with strict inequality somewhere. Let $v_1(\xi)=U_1(\xi)e^{\frac{c\xi}{2D_1}}$ and $v_2(\xi)=U_2(\xi)e^{\frac{c\xi}{2D_2}}$.
Then if we can choose $a>0$ such that
\begin{equation} \label{eq_*1}
av_1(\xi)=aU_1(\xi)e^{\frac{c\xi}{2D_1}} \leq U_2(\xi)e^{\frac{c\xi}{2D_2}}=v_2(\xi)
\end{equation}
and also
\begin{equation}\label{eq_*2}
\hat{g}_2  \leq a \hat{h}_2 e^{\frac{c\xi}{2}(\frac{1}{D_1}-\frac{1}{D_2})}
\end{equation}
and not both are equalities, then \eqref{eq_aim1_gU2_hU1} will hold and we reach the conclusion.

In order to choose suitable $a$, we note that $v_1$ and $v_2$ satisfy
\begin{equation}
0= D_1 v_1'' +\left(r_1 -\frac{c^2}{4D_1} -\hat{h}_1^n e^{-\frac{c\xi}{2D_1}n} v_1^n\right)v_1,
\end{equation}
\begin{equation}
0= D_2 v_2'' +\left(r_2 -\frac{c^2}{4D_2} -\hat{h}_2^n e^{-\frac{c\xi}{2D_2}n} v_2^n\right)v_2.
\end{equation}
Therefore, $av_1$ will be a subsolution for $v_2$ (so \eqref{eq_*1} holds) as long as
\begin{align}
0 &\leq D_2 v_1'' +\left(r_2 -\frac{c^2}{4D_2} -\hat{h}_2^n a^n e^{-\frac{c\xi}{2D_2}n} v_1^n\right)v_1 \\
 &= - \frac{D_2}{D_1} \left(r_1 -\frac{c^2}{4D_1} -\hat{h}_1^n e^{-\frac{c\xi}{2D_1}n} v_1^n\right)v_1 +\left(r_2 -\frac{c^2}{4D_2} -\hat{h}_2^n a^n e^{-\frac{c\xi}{2D_2}n} v_1^n\right)v_1 \\
&= \left(r_2 -\frac{c^2}{4D_2} - \frac{D_2}{D_1} \left(r_1 -\frac{c^2}{4D_1}\right) -\left( \hat{h}_2^n a^n e^{-\frac{c\xi}{2D_2}n} - \frac{D_2}{D_1} \hat{h}_1^n e^{-\frac{c\xi}{2D_1}n}\right)  v_1^n \right) v_1. \label{eq_subsol_inequality}
\end{align}

\begin{enumerate}
\item
First let us consider the case $c=0$, so $v_2=U_2$, $v_1=U_1$. In this special case $c=0$, condition \eqref{eq_subsol_inequality} becomes that
\begin{align}
0 \leq \left(r_2 - \frac{D_2}{D_1} r_1 -\left( \hat{h}_2^n a^n -\frac{D_2}{D_1}  \hat{h}_1^n \right)  U_1^n \right) U_1.
\end{align}
This is satisfied if $\frac{r_2}{r_1} \geq \frac{D_2}{D_1}$ and if we choose
\begin{equation}\label{eq_a}
a= \left(\frac{r_2}{r_1}\right)^{\frac{1}{n}}\frac{\hat{h}_1}{\hat{h}_2},
\end{equation}
since then
\begin{equation}
r_2 - \frac{D_2}{D_1} r_1 -\left(\hat{h}_2^n a^n -\frac{D_2}{D_1}  \hat{h}_1^n \right)U_1^n
=\left(r_2 - \frac{D_2}{D_1} r_1\right) \left(1-\frac{\hat{h}_1^n}{r_1}U_1^n\right) \geq 0.
\end{equation}
Here we have used the fact that, due to the assumptions \eqref{assumptions_on_F_1_2} and \eqref{eq_gh}, we have
\begin{equation}
U_1\leq K_1 = \frac{r_1^{\frac{1}{n}}}{\hat{h}_1}.
\end{equation}
So, if $\frac{r_2}{r_1} \geq \frac{D_2}{D_1}$ and $a$ is given by equation \eqref{eq_a} then we have \eqref{eq_*1}. Since $c=0$, the condition \eqref{eq_*2} will also be satisfied for this choice of $a$ as long as
\begin{equation}\label{eq_rhg}
\left(\frac{r_2}{r_1}\right)^{\frac{1}{n}}\hat{h}_1 \geq \hat{g}_2.
\end{equation}

For the non-invasibility result we need to reverse the inequalities in \eqref{eq_*1} and \eqref{eq_*2}. Therefore we can get this by exactly the same proof but with the opposite inequalities:
$\frac{r_2}{r_1} \leq \frac{D_2}{D_1}$ and $\left(\frac{r_2}{r_1}\right)^{\frac{1}{n}}\hat{h}_1 \leq \hat{g}_2$.

\item Next allow $c$ to be either zero or non-zero, but suppose that $D_1=D_2 =D$.
Then, condition \eqref{eq_subsol_inequality} becomes that
\begin{align}
0&\leq \left(r_2  - r_1 -\left( \hat{h}_2^n a^n - \hat{h}_1^n \right)  U_1^n \right) v_1.
\end{align}
This is satisfied if $\frac{r_2}{r_1} \geq 1$ and if we choose $a$ according to equation \eqref{eq_a}, since then
\begin{equation}
r_2 -r_1 -\left(\hat{h}_2^n a^n - \hat{h}_1^n \right)U_1^n
=\left(r_2 -  r_1\right) \left(1-\frac{\hat{h}_1^n}{r_1}U_1^n\right) \geq 0.
\end{equation}
Since $D_1=D_2$, the condition \eqref{eq_*2} will also be satisfied for this choice of $a$ as long as \eqref{eq_rhg} holds.

For the non-invasibility result we need to reverse the inequalities in \eqref{eq_*1} and \eqref{eq_*2}. Therefore we can get this by exactly the same proof but with the opposite inequalities.

\item Finally we consider the case $c\neq 0$, $D_1 \neq D_2$, and we wish to find conditions such that \eqref{eq_subsol_inequality} holds.
Certainly this inequality will hold if both
\begin{equation}\label{eq_ineq_for_rD}
r_2 -\frac{c^2}{4D_2} - \frac{D_2}{D_1} \left(r_1 -\frac{c^2}{4D_1}\right) \geq 0
\end{equation}
and
\begin{equation}\label{eq_ineq_for_a1}
-\hat{h}_2^n a^n e^{-\frac{c\xi}{2D_2}n} + \frac{D_2}{D_1} \hat{h}_1^n e^{-\frac{c\xi}{2D_1}n} \geq 0.
\end{equation}
The condition \eqref{eq_ineq_for_a1} can be written as
\begin{equation}\label{eq_ineq_for_a2}
a\leq \left(\frac{D_2}{D_1}\right)^{\frac{1}{n}} \frac{\hat{h}_1}{\hat{h}_2} e^{\frac{c\xi}{2}(\frac{1}{D_2}-\frac{1}{D_1})}.
\end{equation}
Now we also need \eqref{eq_*2} to be satisfied, which becomes
\begin{equation}
\frac{\hat{g}_2}{\hat{h}_2} e^{\frac{c\xi}{2}(\frac{1}{D_2}-\frac{1}{D_1})} \leq a.
\end{equation}
Therefore, if \eqref{eq_ineq_for_rD} holds and if we can choose $a$ to satisfy
\begin{equation}
\frac{\hat{g}_2}{\hat{h}_2} e^{\frac{c\xi}{2}(\frac{1}{D_2}-\frac{1}{D_1})} \leq a \leq \left(\frac{D_2}{D_1}\right)^{\frac{1}{n}} \frac{\hat{h}_1}{\hat{h}_2} e^{\frac{c\xi}{2}(\frac{1}{D_2}-\frac{1}{D_1})}
\end{equation}
for all $0\leq \xi\leq L_0$, then we have both \eqref{eq_*1} and \eqref{eq_*2}, and we reach the conclusion. We can choose $a$ as required provided that
\begin{equation}
\frac{\hat{g}_2}{\hat{h}_1} \max_{[0,L_0]} e^{\frac{c\xi}{2}(\frac{1}{D_2}-\frac{1}{D_1})} \leq \left(\frac{D_2}{D_1}\right)^{\frac{1}{n}} \min_{[0,L_0]} e^{\frac{c\xi}{2}(\frac{1}{D_2}-\frac{1}{D_1})}.
\end{equation}
It is straightforward to calculate that
\begin{equation}
\frac{\min_{[0,L_0]} e^{\frac{c\xi}{2}(\frac{1}{D_2}-\frac{1}{D_1})} }{ \max_{[0,L_0]} e^{\frac{c\xi}{2}(\frac{1}{D_2}-\frac{1}{D_1})}}=
e^{-\frac{L_0}{2}\left\vert c (\frac{1}{D_2}-\frac{1}{D_1})\right\vert},
\end{equation}
and so we can choose $a$ as required provided that
\begin{align}\label{eq_ineq_a_final}
\frac{\hat{g}_2}{\hat{h}_1} \leq \left(\frac{D_2}{D_1}\right)^{\frac{1}{n}} e^{-\frac{L_0}{2}\left\vert c (\frac{1}{D_2}-\frac{1}{D_1})\right\vert}.
\end{align}
Overall, this means that if \eqref{eq_ineq_for_rD} and \eqref{eq_ineq_a_final} hold, then $u_2$ can invade $U_1(\xi)$.

For the non-invasibility result (i.e. to ensure that $u_2$ decays to zero), we need to reverse the inequalities in \eqref{eq_*1} and \eqref{eq_*2}.
Therefore, as well as reversing the inequality from \eqref{eq_ineq_for_rD},
we now require $a$ such that
\begin{equation}
\frac{\hat{g}_2}{\hat{h}_2} e^{\frac{c\xi}{2}(\frac{1}{D_2}-\frac{1}{D_1})} \geq a \geq \left(\frac{D_2}{D_1}\right)^{\frac{1}{n}} \frac{\hat{h}_1}{\hat{h}_2} e^{\frac{c\xi}{2}(\frac{1}{D_2}-\frac{1}{D_1})}
\end{equation}
for all $0\leq \xi\leq L_0$.
We can choose $a$ as required provided that
\begin{equation}
\frac{\hat{g}_2}{\hat{h}_1} \min_{[0,L_0]} e^{\frac{c\xi}{2}(\frac{1}{D_2}-\frac{1}{D_1})} \geq \left(\frac{D_2}{D_1}\right)^{\frac{1}{n}} \max_{[0,L_0]} e^{\frac{c\xi}{2}(\frac{1}{D_2}-\frac{1}{D_1})},
\end{equation}
which becomes the requirement that
\begin{align}
\frac{\hat{g}_2}{\hat{h}_1} \geq \left(\frac{D_2}{D_1}\right)^{\frac{1}{n}} e^{\frac{L_0}{2}\left\vert c (\frac{1}{D_2}-\frac{1}{D_1})\right\vert}.
\end{align}
\end{enumerate}
\end{proof}

To conclude the paper, we apply Proposition \ref{proposition_hu2_gu1} in a similar way but for the case of a periodic interval.
\begin{corollary}\label{corollary_2}
Consider the domain $A(t)<x<A(t)+L(t)$ where $A(t)$ and $L(t)$ are $T$-periodic and are not both constants. Let $f_1$, $f_2$ satisfy assumptions \eqref{assumptions_on_f1}, \eqref{assumptions_on_F_1_1}, \eqref{assumptions_on_F_1_2} \eqref{assumptions_on_f2}, \eqref{assumptions_on_F_2_1} and \eqref{assumptions_on_F_2_2}, and suppose that
$h_1$, $g_2$ and $h_2$ are given by \eqref{eq_gh} for some $n>0$ and positive constants $\hat{h}_1$, $\hat{g}_2$ and $\hat{h}_2$.

Assume that $r_1 >\mu_1$, $r_2 > \mu_2$, and let $u_1^{*}(\xi,t)$ and $u_2^{*}(\xi,t)$ be as above.
\begin{enumerate}
\item Suppose $D_1=D_2$. \\
If $r_2 \geq r_1$ and $\left(\frac{r_2}{r_1}\right)^{\frac{1}{n}}\hat{h}_1 \geq \hat{g}_2$, and at least one of these inequalities is strict, then $u_2$ can invade $(u_1^{*},0)$.\\
If $r_2 \leq r_1$ and $\left(\frac{r_2}{r_1}\right)^{\frac{1}{n}}\hat{h}_1 \leq \hat{g}_2$, and at least one of these inequalities is strict, then $u_2$ cannot invade $(u_1^{*},0)$.
\item
Suppose $D_1\neq D_2$. \\
Let $v_1(\xi,t)=u_1^{*}(\xi,t)\exp\left( \frac{E(\xi,t)}{D_1}\right)$ and
$v_2(\xi,t)=u_2^{*}(\xi,t)\exp\left( \frac{E(\xi,t)}{D_2}\right)$ where
\begin{equation}\label{eq_E}
E(\xi,t)=\frac{\dot{L}(t)L(t)\xi^2}{4 L_0^2} +\frac{\dot{A}(t)L(t)\xi}{2 L_0}.
\end{equation}
Let $\Delta E = \max_{[0,L_0]\times [0,T]}E - \min_{[0,L_0]\times [0,T]}E$. If both
\begin{align}
\left(1- \frac{D_2}{D_1}\right)\frac{\frac{\partial v_1}{\partial t}}{v_1}\leq r_2 & -\frac{D_2}{D_1} r_1 -\left(1- \frac{D_2}{D_1}\right)\frac{\dot{L}(t)}{2L(t)}  \nonumber\\ & +\left(\frac{1}{D_2} - \frac{D_2}{D_1^2}\right)\left(-\frac{\dot{A}(t)^2}{4}+\frac{\ddot{L}(t)L(t)\xi^2}{4 L_0^2} + \frac{\ddot{A}(t)L(t)\xi}{2 L_0}\right)
\end{align}
for all $0\leq \xi\leq L_0$, $0\leq t\leq T$ and
\begin{align}
\frac{\hat{g}_2}{\hat{h}_1}\leq  \left(\frac{D_2}{D_1}\right)^{\frac{1}{n}}\exp\left( -\left\vert \frac{1}{D_2}-\frac{1}{D_1}\right\vert
\Delta E \right),
\end{align}
then $u_2$ can invade $(u_1^{*},0)$.\\
If both
\begin{align}
\left(1- \frac{D_2}{D_1}\right)\frac{\frac{\partial v_1}{\partial t}}{v_1}\geq r_2 & -\frac{D_2}{D_1} r_1 -\left(1- \frac{D_2}{D_1}\right)\frac{\dot{L}(t)}{2L(t)} \nonumber\\ & +\left(\frac{1}{D_2} - \frac{D_2}{D_1^2}\right)\left( -\frac{\dot{A}(t)^2}{4} + \frac{\ddot{L}(t)L(t)\xi^2}{4 L_0^2} + \frac{\ddot{A}(t)L(t)\xi}{2 L_0}\right)
\end{align}
for all $0\leq \xi\leq L_0$, $0\leq t\leq T$ and
\begin{align}
\frac{\hat{g}_2}{\hat{h}_1} \geq \left(\frac{D_2}{D_1}\right)^{\frac{1}{n}}\exp\left(\ \left\vert \frac{1}{D_2}-\frac{1}{D_1}\right\vert
\Delta E\right),
\end{align}
then $u_2$ cannot invade $(u_1^{*},0)$.
\end{enumerate}
\end{corollary}
\begin{proof}
We shall show that the conditions from Proposition \ref{proposition_hu2_gu1} hold.
Since $g_2$ and $h_2$ are given by \eqref{eq_gh}, what we need to show for the invasibility is that
\begin{equation}\label{eq_aim1_gu*2_hu*1}
\hat{g}_2 u_1^{*}(\xi,t) \leq \hat{h}_2 u_2^{*}(\xi,t) \qquad \textrm{for all }
0\leq \xi\leq L_0, \ 0\leq t\leq T,
\end{equation}
and with strict inequality somewhere. This will hold if we can choose $a>0$ such that both
\begin{equation} \label{eq_*1*}
a v_1(\xi,t) \leq v_2(\xi,t)
\end{equation}
and also
\begin{equation}\label{eq_*2*}
\hat{g}_2 \leq a \hat{h}_2 \exp\left( E(\xi,t) \left(\frac{1}{D_1}-\frac{1}{D_2}\right) \right)
\end{equation}
on $[0,L_0]\times [0,T]$ and with strict inequality somewhere.
In order to choose suitable $a$, we note that $v_1$ and $v_2$ satisfy
\begin{align}\label{eq_v1}
\frac{\partial v_1}{\partial t}
=\frac{D_1 L_0^2}{L(t)^2}\frac{\partial^2 v_1}{\partial \xi^2} +\biggl(r_1 & -\frac{\dot{A}(t)^2}{4D_1} -\frac{\dot{L}(t)}{2L(t)} +\frac{\ddot{L}(t)L(t)\xi^2}{4D_1 L_0^2} + \frac{\ddot{A}(t)L(t)\xi}{2D_1 L_0} \biggr)v_1 \nonumber\\
-& \hat{h}_1^n e^{-\frac{n E(\xi,t)}{D_1}} v_1^{n+1}
\end{align}
and
\begin{align}
\frac{\partial v_2}{\partial t} = \frac{D_2 L_0^2}{L(t)^2}\frac{\partial^2 v_2}{\partial \xi^2} +\biggl(r_2 & -\frac{\dot{A}(t)^2}{4D_2} -\frac{\dot{L}(t)}{2L(t)} +\frac{\ddot{L}(t)L(t)\xi^2}{4D_2 L_0^2} + \frac{\ddot{A}(t)L(t)\xi}{2D_2 L_0}\biggr)v_2 \nonumber\\
- &  \hat{h}_2^n e^{-\frac{n E(\xi,t)}{D_2}} v_2^{n+1}.
\end{align}
Therefore, $av_1$ will be a subsolution for $v_2$ (and so \eqref{eq_*1*} holds) as long as
\begin{align}
\frac{\partial v_1}{\partial t} \leq \frac{D_2 L_0^2}{L(t)^2}\frac{\partial^2 v_1}{\partial \xi^2} +\biggl(r_2 & -\frac{\dot{A}(t)^2}{4D_2} -\frac{\dot{L}(t)}{2L(t)} +\frac{\ddot{L}(t)L(t)\xi^2}{4D_2 L_0^2} + \frac{\ddot{A}(t)L(t)\xi}{2D_2 L_0}\biggr)v_1 \nonumber\\
-&\hat{h}_2^n a^n e^{-\frac{n E(\xi,t)}{D_2}} v_1^{n+1}.
\end{align}
Using the equation \eqref{eq_v1} for $v_1$ to replace the term involving $\frac{\partial^2 v_1}{\partial \xi^2}$, this becomes the requirement that
\begin{align}
\left(1- \frac{D_2}{D_1}\right)\frac{\partial v_1}{\partial t} \leq& \left(r_2 -\frac{D_2}{D_1} r_1 -\left(1- \frac{D_2}{D_1}\right)\frac{\dot{L}(t)}{2L(t)} \right) v_1 \nonumber\\
& +\left(\frac{1}{D_2} - \frac{D_2}{D_1^2}\right)\left(-\frac{\dot{A}(t)^2}{4} +\frac{\ddot{L}(t)L(t)\xi^2}{4 L_0^2} + \frac{\ddot{A}(t)L(t)\xi}{2 L_0}\right) v_1 \nonumber\\
& -\left(\hat{h}_2^n a^n e^{-\frac{n E(\xi,t)}{D_2}}
- \frac{D_2}{D_1}\hat{h}_1^n e^{-\frac{n E(\xi,t)}{D_1}} \right) v_1^{n+1}. \label{eq_subsol*_inequality}
\end{align}
\begin{enumerate}
\item
In the special case $D_1=D_2=D$, the inequality \eqref{eq_subsol*_inequality} becomes that
\begin{align}
0 &\leq \left(r_2 - r_1 -\left(\hat{h}_2^n a^n - \hat{h}_1^n \right) u_1^{*n}\right)v_1.
\end{align}
This will be satisfied if $r_2\geq r_1$ and we chose $a$ according to equation \eqref{eq_a}. Indeed, in that case
\begin{align}
r_2 -r_1 -\left(\hat{h}_2^n a^n - \hat{h}_1^n \right)u_1^{*n}
&=\left(r_2 -  r_1\right) \left(1-\frac{\hat{h}_1^n}{r_1}u_1^{*n}\right) \geq 0,
\end{align}
since due to the assumptions \eqref{assumptions_on_F_1_2} and \eqref{eq_gh}, we have
$u_1^{*}\leq K_1 = \frac{r_1^{\frac{1}{n}}}{\hat{h}_1}$.
Now we also need \eqref{eq_*2*} to be satisfied. In the case $D_1=D_2$ and with $a$ given by equation \eqref{eq_a}, this becomes condition \eqref{eq_rhg}.

For the non-invasibility result (i.e. to ensure that $u_2$ decays to zero), we need to reverse the inequalities in \eqref{eq_*1*} and \eqref{eq_*2*}.
Therefore we can get this by exactly the same proof but with the opposite inequalities.
\item
Now suppose that $D_1\neq D_2$.
Note that (by using Hopf's Lemma and L'H{\^{o}}pital's rule) the ratio $\frac{\frac{\partial v_1}{\partial t}}{v_1}$ is bounded on $[0,L_0]\times [0,T]$. So, certainly the inequality \eqref{eq_subsol*_inequality} will hold if both
\begin{align}\label{eq_ineq_for_rD*}
\left(1- \frac{D_2}{D_1}\right)\frac{\frac{\partial v_1}{\partial t}}{v_1}\leq & r_2  -\frac{D_2}{D_1} r_1 -\left(1- \frac{D_2}{D_1}\right)\frac{\dot{L}(t)}{2L(t)} \nonumber\\
& +\left(\frac{1}{D_2} - \frac{D_2}{D_1^2}\right)\left(-\frac{\dot{A}(t)^2}{4}+\frac{\ddot{L}(t)L(t)\xi^2}{4 L_0^2} + \frac{\ddot{A}(t)L(t)\xi}{2 L_0}\right)
\end{align}
and
\begin{equation}\label{eq_a*1}
\hat{h}_2^n a^n e^{-\frac{n E(\xi,t)}{D_2}}
- \frac{D_2}{D_1}\hat{h}_1^n e^{-\frac{n E(\xi,t)}{D_1}} \leq 0
\end{equation}
on $[0,L_0]\times [0,T]$. The condition \eqref{eq_a*1} can be written as
\begin{equation}
 a\leq \left(\frac{D_2}{D_1}\right)^{\frac{1}{n}} \frac{\hat{h}_1}{\hat{h}_2} \exp \left( E(\xi,t) \left(\frac{1}{D_2}-\frac{1}{D_1}\right) \right).
\end{equation}
Now we also need \eqref{eq_*2*} to be satisfied, and therefore we need to choose $a$ to satisfy
\begin{equation}
\frac{\hat{g}_2}{\hat{h}_2} \exp\left( E(\xi,t) \left(\frac{1}{D_2}-\frac{1}{D_1}\right)\right) \leq a \leq\left(\frac{D_2}{D_1}\right)^{\frac{1}{n}} \frac{\hat{h}_1}{\hat{h}_2} \exp \left( E(\xi,t) \left(\frac{1}{D_2}-\frac{1}{D_1}\right) \right)
\end{equation}
on $[0,L_0]\times [0,T]$. We can choose $a$ as required provided that
\begin{equation}
\frac{\hat{g}_2}{\hat{h}_1}\left(\frac{D_1}{D_2}\right)^{\frac{1}{n}} \leq
\frac{\min_{[0,L_0]\times [0,T]}  \exp\left( E(\xi,t)\left(\frac{1}{D_2}-\frac{1}{D_1}\right)\right) }{ \max_{[0,L_0]\times [0,T]}  \exp\left(E(\xi,t)\left(\frac{1}{D_2}-\frac{1}{D_1}\right)\right)}.
\end{equation}
This condition can be written as
\begin{align} \label{eq_ineq_a_final*}
\frac{\hat{g}_2}{\hat{h}_1} \left(\frac{D_1}{D_2}\right)^{\frac{1}{n}}\leq \exp\left( -\left\vert \frac{1}{D_2}-\frac{1}{D_1}\right\vert \Delta E \right).
\end{align}
Overall, this means that if $D_1\neq D_2$, and \eqref{eq_ineq_a_final*} is satisfied and \eqref{eq_ineq_for_rD*} holds on $[0,L_0]\times [0,T]$, then $u_2$ can invade $(u_1^{*},0)$.

For the non-invasibility result (i.e. to ensure that $u_2$ decays to zero), we need to reverse the inequalities in \eqref{eq_*1*} and \eqref{eq_*2*}.
Therefore, as well as reversing the inequality from \eqref{eq_ineq_for_rD*},
we now require $a$ such that
\begin{equation}
\frac{\hat{g}_2}{\hat{h}_2} \exp\left( E(\xi,t)\left(\frac{1}{D_2}-\frac{1}{D_1}\right)\right) \geq a \geq\left(\frac{D_2}{D_1}\right)^{\frac{1}{n}} \frac{\hat{h}_1}{\hat{h}_2} \exp \left( E(\xi,t) \left(\frac{1}{D_2}-\frac{1}{D_1}\right) \right)
\end{equation}
on $[0,L_0]\times [0,T]$. We can choose $a$ as required provided that
\begin{equation}
\frac{\hat{g}_2}{\hat{h}_1}\left(\frac{D_1}{D_2}\right)^{\frac{1}{n}} \geq
\frac{\max_{[0,L_0]\times [0,T]} \exp\left( E(\xi,t)\left(\frac{1}{D_2}-\frac{1}{D_1}\right)\right) }{ \min_{[0,L_0]\times [0,T]} \exp\left(E(\xi,t)\left(\frac{1}{D_2}-\frac{1}{D_1}\right)\right)},
\end{equation}
which becomes the condition that
\begin{align}
\frac{\hat{g}_2}{\hat{h}_1} \left(\frac{D_1}{D_2}\right)^{\frac{1}{n}}\geq \exp\left(\ \left\vert \frac{1}{D_2}-\frac{1}{D_1}\right\vert
 \Delta E \right).
\end{align}
\end{enumerate}
\end{proof}
\section*{Acknowledgements}
This work was funded by EPSRC (reference EP/W522545/1). I would also like to thank Professor Elaine Crooks for her useful discussions.

\end{document}